\documentclass[11pt]{amsart}
\usepackage{amsmath}
\usepackage{amsthm}
\usepackage{amssymb}
\usepackage{amsfonts,mathrsfs}
\usepackage{stmaryrd}
\usepackage{amsxtra}  
\usepackage{epsfig}
\usepackage{verbatim}
\usepackage[all]{xy}

\usepackage{mathtools, hyperref}
\usepackage{tikz}
\usetikzlibrary{shapes}
\usetikzlibrary{plotmarks}
\usepackage{bm}

\theoremstyle{plain}
\newtheorem{theorem}[equation]{Theorem}
\newtheorem{proposition}[equation]{Proposition}
\newtheorem{lemma}[equation]{Lemma}
\newtheorem{corollary}[equation]{Corollary}

\theoremstyle{remark}
\newtheorem{remark}[equation]{Remark}
\numberwithin{equation}{section}

\newcommand{\dee}{\partial}

\newcommand{\wt}{\widetilde}

\def\norm#1{\left\Vert#1\right\Vert}

\newcommand{\ca}{{\mathcal A}}

\newcommand{\cd}{{\mathcal D}}

\newcommand{\ck}{{\mathcal K}}

\newcommand{\co}{{\mathcal O}}

\newcommand{\cR}{{\mathcal R}}
\newcommand{\cs}{{\mathcal S}}
\newcommand{\ct}{{\mathcal T}}

\newcommand{\C}{{\mathbb C}}

\newcommand{\h}{{\mathbb H}}

\newcommand{\Q}{{\mathbb Q}}
\newcommand{\R}{{\mathbb R}}

\newcommand{\Z}{{\mathbb Z}}

\begin{document}

\title[Bergman and Sobolev]{Sobolev mapping of some holomorphic projections}
\author{L. D. Edholm \& J. D. McNeal}
\begin{abstract} Sobolev irregularity of the Bergman projection on a family of domains containing the Hartogs triangle is shown. 
On the Hartogs triangle itself, a sub-Bergman projection is shown to satisfy better Sobolev norm estimates
than its Bergman projection.

\end{abstract}
\address{Department of Mathematics, \newline University of Michigan, Ann Arbor, Michigan, USA}
\email{edholm@umich.edu}

\address{Department of Mathematics, \newline The Ohio State University, Columbus, Ohio, USA}
\email{mcneal@math.ohio-state.edu}

\maketitle 

\section{Introduction}\label{S:Intro}

If $\Omega\subset\C^n$ is an open set, $1<p<\infty$, and $k\in\Z^+$, let $L^p_k(\Omega)$ denote the usual $L^p$ Sobolev space of order $k$: the measurable functions $f$
such that 
\begin{equation*}\label{D:LpSobolevNorm}
\norm{f}_{L_k^p(\Omega)} = \left( \sum_{|\alpha|\le k} \int_{\Omega} |\dee^\alpha f|^p\,dV \right)^{\frac1p}
\end{equation*} 
is finite, where derivatives are interpreted in the distributional sense. 

This paper continues investigations from  \cite{EdhMcN16}, \cite{EdhMcN17} by demonstrating irregularity in the $L^p$ Sobolev spaces for the Bergman projection associated to domains defined in \eqref{D:GenHartogs} below. These generalize the Hartogs triangle, which is $\h_1$ in \eqref{D:GenHartogs}.

The Bergman projection, $\bm{B} =\bm{B}_\Omega$, orthogonally projects $L^2(\Omega)$ onto the closed subspace $\co(\Omega)\cap L^2(\Omega)$,  $\co(\Omega)$ denoting holomorphic functions.
On $L^2(\Omega)=L^2_0(\Omega)$, $\bm{B}$ is represented as an integral operator
\begin{equation}\label{D:Bdef}
\bm{B}f(z)=\int_\Omega B_\Omega (z,w) f(w)\, dV(w),\qquad f\in L^2(\Omega),
\end{equation}
where $dV$ denotes Lebesgue measure and $B_\Omega(z,w)\in \co(\Omega)\times\overline{\co(\Omega)}$ is the Bergman kernel. If $f\notin L^2(\Omega)$
let \eqref{D:Bdef} {\it define} $\bm{B}f$, whenever the integral converges.  For many classes of pseudoconvex domains, precise pointwise estimates on $B_\Omega(z,w)$ were obtained and shown to imply $\left\|\bm{B}f\right\|_{L^p_k(\Omega)}\leq C\|f\|_{L^p_k(\Omega)}$  for all $1<p<\infty$ and $k\in\Z^+$. See \cite{ChaDup06, Koenig02, McNeal89, McNeal91, McNeal94, NagRosSteWai89, PhoSte77}. 
Thus $\bm{B}$ is $L^p_k$-regular in these cases. In the special case $p=2$, regularity for all $k\in\Z^+$ was shown in \cite{BoaStr91} whenever $\Omega$ has a plurisubharmonic defining function, without establishing pointwise estimates on $B_\Omega(z,w)$. This result was generalized in \cite{BoasStraube99, Kohn99}. However $L^2_k$ regularity does not always hold.  This irregularity was discovered 
in connection to Condition $R$ of Bell-Ligocka \cite{BellLigocka, Bell81a}: it is shown in \cite{Bar92}  that $\bm{B}$ is irregular on $L^2_k(W)$, for large $k$, on the 
pseudoconvex ``worm'' domains $W$ given in 
 \cite{DieFor77-1}. 

The irregularity of  $\bm{B}$ demonstrated in  \cite{EdhMcN16, EdhMcN17} is somewhat different. It occurs on the Lebesgue spaces $L^p(\Omega)=L^p_0(\Omega)$ for certain $p\neq 2$ and does not involve derivatives. For $\gamma >0$, define
\begin{equation}\label{D:GenHartogs}
\h_{\gamma} = \{(z_1,z_2)\in\C^2 : |z_1|^\gamma<|z_2|<1 \}. 
\end{equation}
In \cite{EdhMcN17}  it is shown that  the Bergman projection on $\h_\gamma$, for any $\gamma$, is a degenerate $L^p$ operator, bounded only for $p$ in a proper subinterval of $(1,\infty)$. In particular the situation
on $\h_1$ is that $\bm{B}:L^p\left(\h_1\right)\to L^p\left(\h_1\right)$ boundedly if and only if $\frac 43 <p <4$; see Theorem \ref{T:JGALpBoundedness} below for the general situation. The limited range of $L^p$ boundedness 
has consequences for approximation and duality theory in $\co\left(\h_\gamma\right)$, see \cite{ChaEdhMcN18}. Similar consequences hold when irregularity can be characterized on norm scales other than $L^p$.

It turns out $\bm{B}$ is very degenerate as an $L^p$ Sobolev map. 
  
\begin{theorem}\label{T:Main} Let  $\gamma = \frac mn \in \Q^+$ and $\bm{B}$ denote the Bergman projection on $\h_\gamma$.  
\begin{itemize}
\item[(1)] $\bm{B}$ fails to map $L^2_k(\h_\gamma) \to L^2_k(\h_\gamma)$, for $k\ge1$ an integer.
\medskip
\item[(2)] Let $j,l\in\Z^+$.  Then $\frac{\dee^{j+l}}{\dee z_1^j \dee z_2^l} \circ \bm{B}$ fails to map $C^{\infty}(\overline{\h_{\gamma}}) \to L^p(\h_\gamma)$ for 
\begin{equation*}
p \ge \frac{2m+2n}{m(l+1) + n(j+1)-1}.
\end{equation*}
\end{itemize}
\end{theorem}

$\bm{B}$ also exhibits some regularity in $L^p$ Sobolev norms, but only on $\h_1$:

\begin{theorem}\label{T:Main2}
 $\bm{B}$ maps $L^p_1(\h_1) \to L^p_1(\h_1)$ boundedly, for $\frac{4}{3}<p<2$.
\end{theorem}

Notice the range on $p$, for boundedness on $L^p_1(\h_1)$, is smaller than the range for boundedness on $L^p_0(\h_1)$. More general statements than Theorem  \ref{T:Main2}
can be made -- for separate directional derivatives and on domains other than $\h_1$ -- but these do not yield boundedness theorems on the full Sobolev spaces $L^p_k$; see Section \ref{S:PositiveMapping}.

Proving Theorems \ref{T:Main} and \ref{T:Main2} requires understanding how derivatives commute past the Bergman projection.
An initial difficulty is that $\h_\gamma$ is not smoothly bounded, so Stokes' theorem cannot be applied in the usual way, e.g., as in \cite[Lemma 3]{McNSte94}, \cite[Proposition 3.3]{McNSte97}, or \cite[Lemma 5.1]{NagRosSteWai89}.
We circumvent this by applying Stokes theorem on appropriately chosen discs and annuli intersecting $\h_\gamma$. Theorem \ref{T:Main} is proved in Section \ref{S:NegativeMapping}. After developing some general tools,  
Theorem \ref{T:Main2} is proved as Corollary \ref{C:regularity} in Section \ref{S:PositiveMapping}.  To partially repair  irregularity described by Theorem \ref{T:Main}, substitute operators related to $\bm{B}$ are considered in Section \ref{S:SubstituteOperators}.

There are other papers showing  Bergman irregularity on $L^p_0(\Omega)$, for specific pseudoconvex $\Omega$: \cite{BarSah12, ChaZey16, Chen17, Zey13}. A unifying result, explaining irregularity in these cases and  \cite{EdhMcN16, EdhMcN17},  is lacking. A weighted regularity result on $L^p_k(\h_1)$, $k>0$, related to Theorem \ref{T:Main}, was obtained in \cite{Chen17c}.
See also the paper \cite{Bar84} for a nonpseudoconvex domain with $L^p_0$-irregularity of its Bergman projection.

When $X$ and $Y$ are expressions involving several variables, write $X\lesssim Y$ to mean $X\leq C Y$ for a constant $C$ independent of certain of these variables. The independence of which variables is specified in use.
$X\approx Y$ means $X\lesssim Y\lesssim X$ holds.

\section{Sobolev regularity in one variable}\label{S:C1} 
Let $D \subset \C$ denote the unit disc.  The Bergman projection $\bm{B}_D$ is bounded from $L^p_k(D)\to L^p_k(D)$ for all $1<p<\infty$ and $k\in\Z^+$. 
This is well-known when $k=0$, apparently first proved in \cite{ZahJud64} using singular integral operator theory; see \cite[Chapter 2]{DurSch04}. For any $k\in\Z^+$, a proof modeled on arguments in \cite{McNSte94} is given below. This serves as a template for the proof of  Theorem \ref{T:Main2}.

The Bergman kernel of $D$ is 
\begin{equation}\label{E:BergmanKernelDisc}
B_D(z,w) = \frac{1}{\pi} \sum_{j=0}^\infty (j+1) (z\bar{w})^j = \frac{1}{\pi} \frac{1}{(1-z\bar{w})^2}.
\end{equation}
Note $B_D(z,w)$ can be viewed as a function of  $s = z\bar{w}$.

\subsection{$L^p_0$ boundedness}
A family of integral estimates will be used. When $A =0$, the result is often called the Forelli-Rudin lemma; see \cite{ForRud74}, \cite{RudinFunctiontheory}, or \cite{ZhuBergmanbook} for the `standard' proof, based on
asymptotics of the gamma function. Different proofs are given in \cite{EdhMcN16}, \cite{EdhMcN17}, \cite{Chen17}, which also address $A\neq 0$.

\begin{lemma}\label{L:IntegralEstimateDisc}
Let $D\subset\C$ be the unit disc, $\epsilon \in (0,1)$ and $A < 2$.  Then for $z\in D$,
$$\int_D \frac{(1 - |w|^2)^{-\epsilon}}{|1 - z\bar{w}|^2}|w|^{-A} \, dV(w) \lesssim (1 - |z|^2)^{-\epsilon},$$
for a constant $C=C(A, \epsilon)$ independent of $z$.
\end{lemma}

A general version of Schur's Lemma will also be used. The next result extends Lemma 2.4 from \cite{EdhMcN16}.

\begin{lemma}\label{L:GeneralSchur}
Let $\Omega \subset \C^n$ be a domain and $K:\Omega \times\Omega \to [0,\infty)$ a kernel function. Suppose there is an auxiliary function $h:\Omega \to [0,\infty)$ and numbers $0\le \alpha<\beta$, $0\le \gamma<\delta$ such that the following two estimates hold: 
\smallskip
For all $\epsilon \in [\alpha,\beta)$,
\begin{equation}\label{E:Schur1}
\int_\Omega K(z,w)h(w)^{-\epsilon} \, dV(w) \lesssim h(z)^{-\epsilon},
\end{equation}
and for all $\epsilon \in [\gamma,\delta)$,
\begin{equation}\label{E:Schur2}
\int_\Omega K(z,w)h(z)^{-\epsilon} \, dV(z) \lesssim h(w)^{-\epsilon}.
\end{equation}
\smallskip

Then the operator $\ck$, 
$\ck (f)(z) := \int_\Omega K(z,w)f(w)\,dV(w)$, maps
$L^p(\Omega)\to L^p(\Omega)$ for all p in the range 
\begin{equation}\label{E:p_range}
\frac{\gamma}{\beta}+1 < p < \frac{\delta}{\alpha}+1.
\end{equation}
\begin{proof}
Let $\frac1p + \frac1q = 1$, $g \in L^p(\Omega)$ and $s\in[\alpha,\beta)$.  Then
\begin{align*}
\left|\ck(f)(z)\right|^p &\le \left(\int_\Omega K(z,w) |f(w)|^p h(w)^{\frac{sp}{q}} \,dV(w)\right) \left(\int_\Omega K(z,w) h(w)^{-s} \,dV(w)\right)^{\frac pq} \\
&\lesssim \left(\int_\Omega K(z,w) |f(w)|^p h(w)^{\frac{sp}{q}} \,dV(w)\right) h(z)^{-\frac{sp}{q}}. 
\end{align*}
The first inequality follows from H\"older's inequality, the second from \eqref{E:Schur1}. Now
\begin{align}
\int_{\Omega}\left|\ck(f)(z)\right|^pdV(z) &\lesssim \int_\Omega \left(\int_\Omega K(z,w) |f(w)|^p h(w)^{\frac{sp}{q}} \,dV(w)\right) h(z)^{-\frac{sp}{q}} dV(z) \notag \\
&= \int_\Omega |f(w)|^p\, h(w)^{\frac{sp}{q}} \left(\int_\Omega K(z,w) h(z)^{-\frac{sp}{q}} \,dV(z)\right) dV(w) \label{E:UseSchur2}.
\end{align}
When $s\in[\alpha,\beta)$ may chosen so that also $\frac{sp}{q}\in[\gamma,\delta)$, estimate \eqref{E:Schur2} implies
\begin{equation*}
\eqref{E:UseSchur2} \lesssim \int_\Omega |f(z)|^p\,dV(z),
\end{equation*}
and thus $\ck:L^p(\Omega) \to L^p(\Omega)$ boundedly.  The existence of such an $s$  is equivalent to saying {\em both} the inequalities
$\frac{q}{p}\gamma<\beta$ and $\alpha <\frac{q}{p} \delta$
hold. This is equivalent to saying \eqref{E:p_range} holds, as claimed.
\end{proof}
\end{lemma}

Lemmas \ref{L:IntegralEstimateDisc} and \ref{L:GeneralSchur} suffice to show $L^p_0(D)$ boundedness of $\bm{B}_D$.
\begin{corollary}\label{C:LpBergmanDisc}
The Bergman projection $\bm{B}_{D}$ maps $L^p(D)$ to $L^p(D)$ for all $1<p<\infty$. 

In fact, the operator whose kernel is $|B_D(z,w)|$ is bounded on $L^p(D)$ for $1<p<\infty$.
\end{corollary}
\begin{proof}
Lemma \ref{L:GeneralSchur} is used with $K(z,w)= |B_D(z,w)|$ and $h(w) = 1- |w|^2$ as the auxiliary function. Lemma \ref{L:IntegralEstimateDisc} shows that estimate \eqref{E:Schur1} holds for all $0<\epsilon <1$.  Since $B_D(z,w)$ is conjugate symmetric, \eqref{E:Schur1} is equivalent to \eqref{E:Schur2} with $\alpha = \gamma$ and $\beta=\delta$.  Lemma \ref{L:GeneralSchur} then gives the claimed boundedness by setting $\beta=1$ and sending $\alpha \to 0^+$.
\end{proof}

\subsection{Integration by parts}\label{SS:VectorFieldTw}
Define the vector field
\begin{equation}\label{D:VectorFieldTw}
\ct_w = \bar{w}\frac{\dee}{\dee\bar{w}} - w \frac{\dee}{\dee w},
\end{equation}
and write $\ct_w^k$ to mean $\ct_w \circ \dots \circ \ct_w$ composed $k$ times. If $f\in L_k^p(D)$, clearly $\ct_w^kf \in L^p(D)$. If, in addition, $f\in\co(D)$, a partial converse holds: if $D_\delta=\{z\in D: |z|>\delta\}$, then $\|f\|_{L_k^p(D_\delta)}\lesssim \left\|\ct_w^kf\right\|_{L^p(D_\delta)}$ for a constant independent of $f$.\footnote{A version of this also holds in several variables. See, e.g., \cite{Boas87, Barrett86, HerMcN12} for a statement of the result, as well as elementary proofs for $p=2$. For general $p$, see \cite{Detraz81}.} This holds since any first derivative can be written as a linear combination $A\ct_w +B\frac\partial{\partial\bar w}$ on $D_\delta$, for bounded functions $A$ and $B$. 

The crucial property $\ct_w$ satisfies is
\begin{proposition}\label{P:baddirection}
$\ct_w$ annihilates $C^1$ radial functions of $w\in\C$.
\end{proposition}
\begin{proof}
A $C^1$ radial function $g$ can be written as $g(w) = f(|w|^2)$, where $f\in C^1\left([0,\infty)\right)$.  Therefore
\begin{equation*}
\ct_w g = \bar{w} f'(|w|^2)\cdot w - w f'(|w|^2)\cdot \bar{w} \equiv 0.
\end{equation*}
\end{proof}
Recall that $r:\C \to \R$ is a defining function for $\Omega$ if $\{r<0\}=\Omega$ and $|\nabla r(w)|\neq 0$ when $r=0$.
Proposition \ref{P:baddirection} implies, in particular,  that $\ct_w$ annihilates defining functions of discs and annuli centered at the origin along their boundaries. 
An integration by parts result follows:
\begin{proposition}\label{P:IntByParts}
Let $\Omega \subset \C$ be either a disc or an annulus centered at the origin.  Then if $f,g\in L^1_1(\Omega)\cap C\left(\overline\Omega\right)$,
\begin{equation*}
\int_\Omega \ct_w f\cdot g\,dV= -\int_\Omega f\cdot\ct_w g\,dV.
\end{equation*}
\begin{proof}  Choose a defining function for $\Omega$ with $|\nabla r(w)| = 1$ for all $w\in b\Omega$.  Stokes' theorem yields 
\begin{align*} 
\int_\Omega \ct_w f\cdot g\,dV &= \int_\Omega \frac{\dee f}{\dee\bar{w}}(w) \cdot\bar{w} g(w)\,dV(w) - \int_\Omega \frac{\dee f}{\dee w}(w) \cdot w  g(w)\,dV(w) \notag \\
&= -\int_\Omega f\cdot \frac{\dee }{\dee\bar{w}}\big(\bar{w} g(w)\big) \,dV + \int_{b\Omega} f\cdot\bar{w} g \cdot \frac{\dee r}{\dee \bar{w}} \, dS \notag \\
& \qquad \qquad\qquad  + \int_\Omega f\cdot \frac{\dee }{\dee {w}}\big( {w} g(w)\big) \,dV - \int_{b\Omega} f {w} g \cdot \frac{\dee r}{\dee {w}} \, dS \notag \\
&= -\int_\Omega f\cdot \Big[\frac{\dee }{\dee\bar{w}}\big(\bar{w} g(w)\big) - \frac{\dee }{\dee {w}}\big( {w} g(w)\big) \Big]\,dV+\int_{b\Omega} fg\cdot \ct_w r(w)\, dS \notag \\
&= -\int_\Omega f\cdot \ct_wg\,dV.
\end{align*}
Here $dS$ denotes induced surface measure on $b\Omega$.
The last boundary integral vanishes since $\ct_wr\equiv 0$ on $b\Omega$.
\end{proof}
\end{proposition}

\subsection{$L^p_k$ boundedness for $k>0$}\label{SS:MappingBergmanDisc}
\begin{theorem}\label{T:BergmanMappingLpSobolevDisc}
The Bergman projection $\bm{B}_{D}$ is a bounded operator from $L^p_k(D) \to L^p_k(D)$ for all $k\in \Z^+$ and $1<p<\infty$.
\end{theorem}

\begin{proof}
Fix $k,p$ and let $f\in L^p_k(D)$. Since $\bm{B}_{D}f\in\co(\Omega)$, only holomorphic derivatives need to be estimated.  For $z \ne 0$,

\begin{align}
\frac{\dee^k}{\dee z^k}\bm{B}_{D}f(z)  &= \frac{\dee^k}{\dee z^k} \int_D B_D(z,w)f(w)\,dV(w)  \notag \\ 
&= \int_{D} \frac{\dee^k}{\dee z^k} \big( B_D(z,w)\big) f(w)\,dV(w) \notag \\ 
&= \frac{1}{z^k}\int_{D} \bar{w}^k\frac{\dee^k}{\dee \bar{w}^k} \big( B_D(z,w)\big) f(w)\,dV(w). \label{E:WBarDerivative}
\end{align}
The last equality follows because $B_D(z,w)$ can be viewed as a function of the variable $s = z\bar{w}$.  Define a new kernel $K_k(z,w)$, obtained by subtracting away the $(k-1)$-Taylor approximation of $B_D(z,w)$ in the $s$ variable,  i.e.,
\begin{align}
K_k(z,w) := \frac{1}{\pi^2}\left[\frac{1}{(1-s)^2} - \sum_{j=0}^{k-1} (j+1) s^j \right] &= \frac{1}{\pi^2}\frac{\dee}{\dee s}\left[\sum_{j=k}^{\infty} s^{j+1} \right] \notag \\ 
&= \frac{(k+1)s^k-k s^{k+1}}{(1-s)^2} \label{D:K_k(z,w)Disc}.
\end{align}
Since $K_k(z,w)$ and $B_D(z,w)$  differ by terms annihilated by $\frac{\dee^k}{\dee \bar{w}^k}$  and $K_k(z,w)$ is anti-holomorphic in $w$,
\begin{align*}
\eqref{E:WBarDerivative} &= \frac{1}{z^k}\int_{D} \bar{w}^k\frac{\dee^k}{\dee \bar{w}^k} \big( K_k(z,w)\big) f(w)\,dV(w) \notag \\
&= \frac{1}{z^k}\int_{D} \ct_w^k \big( K_k(z,w)\big) f(w)\,dV(w),\notag \\
&= \frac{(-1)^k}{z^k}\int_{D}  K_k(z,w) \ct_w^k f(w)\,dV(w). 
\end{align*}
The last equality follows from Proposition \ref{P:IntByParts}.

The modified kernel $K_k(z,w)$ satisfies a stronger estimate than $B_D(z,w)$.  Indeed, equation \eqref{D:K_k(z,w)Disc} shows
\begin{equation*}
\left|K_k(z,w)\right| \lesssim \frac{|z|^k|w|^k}{|1-z\bar{w}|^2},
\end{equation*}
for a constant independent of $z,w\in D$. This can be used to counteract the factor $\frac 1{z^k}$ appearing in \eqref{E:WBarDerivative}.  Thus
\begin{align}
\left|\frac{\dee^k}{\dee z^k}\bm{B}_{D}f(z)\right| \lesssim \int_D \frac{|w|^k}{|1-z\bar{w}|^2} \ct_w^k f(w)\,dV(w) &\le \int_D \frac{1}{|1-z\bar{w}|^2} \ct_w^k f(w)\,dV(w) \notag \\
&\approx \int_D |B_D(z,w)|\, \ct_w^k f(w)\,dV(w). \label{E:AbsBergTwkf}
\end{align}
Since $\ct_w^k f \in L^p(D)$, Corollary \ref{C:LpBergmanDisc} says that \eqref{E:AbsBergTwkf} defines an $L^p(D)$ function.  This implies $\frac{\dee^k}{\dee z^k}\bm{B}_{D}f(z) \in L^p(D)$.  

For any positive integer $l \le k$, the same argument -- but for the modified  kernel $K_l (z,w)$ -- shows $\frac{\dee^l}{\dee z^l}\bm{B}_{D}f(z) \in L^p(D)$.  Thus $\bm{B}_Df \in L^p_k(D).$
\end{proof}

\section{Sobolev irregularity}\label{S:NegativeMapping}

The starting point is the characterization of $L^p_0$ boundedness of the Bergman projection on
$\h_\gamma$.

\begin{theorem}[\cite{EdhMcN17}]\label{T:JGALpBoundedness}
Let $\h_\gamma$ be defined in  \eqref{D:GenHartogs}, $\bm{B}$ denote the Bergman projection on $\h_\gamma$, and $1<p<\infty$.
\begin{itemize}
\item[(1)] Let $\gamma=\frac{m}{n} \in \Q^+$, with $\gcd(m,n) = 1$.  

Then $\bm{B}: L^p\left(\h_\gamma\right) \to L^p\left(\h_\gamma\right)$ is bounded if and only if $p \in \left(\frac{2m+2n}{m+n+1}, \frac{2m+2n}{m+n-1} \right)$.
\item[(2)] Let $\gamma > 0$ be irrational.  

Then $\bm{B}: L^p(\h_{\gamma}) \to L^p(\h_{\gamma})$ is bounded if and only if $p=2$.
\end{itemize}
\end{theorem}

Let $\big(\lambda (m,n), \rho (m,n)\big)=  \left(\frac{2m+2n}{m+n+1}, \frac{2m+2n}{m+n-1} \right)$ denote the interval of $L^p$ boundedness in (1) above. When $\h_{m/n}$ is fixed, denote this also as $I^p_0$.

Some ingredients in the proof of Theorem \ref{T:JGALpBoundedness} are used to prove the irregularity statements in Theorem \ref{T:Main}. Only consider $\h_{m/n}$ and let $\bm{B}=\bm{B}_{\h_{m/n}}$. An index $(\alpha_1, \alpha_2)\in\Z^+\times\Z$ is $L^p$-{\it allowable} if
the monomial $z_1^{\alpha_1} \tilde z_2^{\alpha_2}\in L^p\left(\h_{m/n}\right)$, where $\tilde z_2$ is either $z_2$ or $\bar z_2$. This set can be characterized:

\begin{lemma}[\cite{EdhMcN17}, eq. (3.3)]\label{L:(p,m/n)-allowable MultiIndices and Norm}
Let $p \in [1,\infty)$.  The $L^p$-allowable indices are
\begin{equation*}\label{E:(p,m/n)-allowable MultiIndices}
\cs\left(\h_{m/n},L^p\right) = \left\{ \alpha=(\alpha_1,\alpha_2) : \alpha_1 \ge 0, \ \ n\alpha_1+m\alpha_2\ge \left\lfloor -\frac{2}{p}(m+n) +1 \right\rfloor \right\}.
\end{equation*}
\end{lemma}
\noindent See also Lemma 4.4 in \cite{ChaEdhMcN18}.
Here $\lfloor x\rfloor =$ the greatest integer $\leq x$. In particular, the $L^2$ monomials are
\begin{equation}\label{E:L2mono}
\cs\left(\h_{m/n},L^2\right)= \left\{(\alpha_{1},\alpha_{2}): \alpha_{1} \ge 0,\ n\alpha_1 + m\alpha_{2} \ge -m -n +1 \right\}.
\end{equation}
As notation for the ray bounding the sets $\cs\left(\h_{m/n},L^p\right)$, let 
\begin{equation*}
\ell\left(\h_{m/n},L^p\right) = \left\{(x,y)\in\R^2: x\geq 0, \,\, \ nx+my= \left\lfloor -\frac{2}{p}(m+n) +1 \right\rfloor \right\}.
\end{equation*}
A consequence of  orthogonality is also essential.

\begin{lemma}[\cite{EdhMcN17}, Proposition 5.1]\label{P:EasyMonomials}
If both $(\beta_1,\beta_2), (\beta_1, -\beta_2)\in\cs\left(\h_{m/n},L^2\right)$, then 
\begin{equation*}
\bm{B}\left(z_1^{\beta_1}\, \bar z_2^{\,\beta_2}\right) = C\, z_1^{\beta_1}\, z_2^{-\beta_2},
\end{equation*}
for a constant $C>0$.
\end{lemma}

The unboundedness statements in Theorem \ref{T:JGALpBoundedness} for $p\notin I^p_0$ (defined above) are proved as follows. Let $p\geq \rho(m,n)$. 
\begin{description}
\item[(A)] Choose $(\beta_1,\beta_2)\in\Z^+\times\Z^+$  with $(\beta_1,-\beta_2)\in \ell\left(\h_{m/n},L^2\right)$.
\item[(B)] Lemma \ref{L:(p,m/n)-allowable MultiIndices and Norm} implies $z_1^{\beta_1} z_2^{-\beta_2}\notin \cs\left(\h_{m/n},L^{\rho(m,n)}\right)$.
\item[(C)] Let $f(z_1, z_2) =: z_1^{\beta_1}\bar z_2^{\beta_2}$; Lemma \ref{P:EasyMonomials} says $\bm{B}f = C z_1^{\beta_1} z_2^{-\beta_2}$. Thus $\|f\|_{L^p} <\infty$, while
$\left\|\bm{B}f\right\|_{L^p}=\infty$. 
\end{description}
\noindent Duality implies the same conclusion if $p\leq\lambda(m,n)$.

The heart of this argument works on Sobolev spaces. But one piece is not transferable: if $j,l\in\Z^+$, the operator $\frac{\dee^{j+l}}{\dee z_1^j \dee z_2^l} \circ \bm{B}$ is not self-adjoint in the
$L^2$ inner product. As a result, knowing that $\bm{B}$ is unbounded on $L^p_k$ does not automatically imply that $\bm{B}$ is unbounded on $L^q_k$, where $\frac 1p +\frac 1q =1$. 
Whether this fact actually allows {\it regularity} of  $\bm{B}$ on $L^q_k$ for small $q$, e.g.,  $q < \lambda(m,n)$, in cases where $\bm{B}$ is unbounded on $L^p_k$
is uncertain.
\medskip

However for large $p$, $L^p_k$ regularity certainly does not hold:

\begin{theorem}\label{T:DerivsJLBerg}
Let $\gamma = \frac mn \in \Q^+$, and $j,l$ non-negative integers.  The operator $\frac{\dee^{j+l}}{\dee z_1^j \dee z_2^l} \circ \bm{B}$ fails to map $C^{\infty}(\overline{\h_{\gamma}}) \to L^p(\h_\gamma)$ for any 
\begin{equation}\label{E:LpBound_jl}
p \ge \frac{2m+2n}{m(l+1) + n(j+1)-1}.
\end{equation}
\end{theorem}

\begin{proof}
Starting from equation \eqref{E:L2mono}, choose $\beta = (\beta_1,\beta_2) \in \Z^+ \times \Z^+$ with $\beta_1\ge j$ and $(\beta_1,-\beta_2)\in \ell\left(\h_{m/n},L^2\right)$, i.e., $n\beta_1-m\beta_2 = 1-m-n$.  Clearly
\begin{equation}
\frac{\dee^{j+l}}{\dee z_1^j \dee z_2^l}\Big(z_1^{\beta_1} z_2^{-\beta_2}\Big) \approx z_1^{\beta_1-j} z_2^{-\beta_2-l}.
\end{equation}
To see when this is an $L^p$ function, compute
\begin{align}
\int_{\h_\gamma} |z_1^{\beta_1-j}z_2^{-\beta_2-l}|^p\,dV &= 4\pi^2 \int_{h_\gamma} r_1^{p\beta_1-pj+1} r_2^{-p\beta_2-pl+1}\,dr_1 dr_2 \notag\\
&= 4\pi^2 \int_0^1 r_2^{-p\beta_2-pl+1} \int_0^{r_2^{n/m}} r_1^{p\beta_1-pj+1} \,dr_1 dr_2 \notag\\
&\approx \int_0^1 r_2^{-p\beta_2-pl+1+pn\beta_1/m - pnj/m +2n/m}\,dr_2 \label{E:IntApprox1},
\end{align}
where $h_\gamma$ is the Reinhardt shadow of $\h_\gamma$, i.e., $h_\gamma = \h_\gamma \cap (\R^{\ge0} \times \R^{\ge0})$.  The integral in \eqref{E:IntApprox1} is finite if and only if the exponent on the integrand $> -1$.  This is equivalent to saying
\begin{align}\label{E:UpperBoundBjl}
p < \frac{2m+2n}{m(l+1)+n(j+1)-1}.
\end{align}

Now consider the monomial $f(z) = z_1^{\beta_1}\bar{z}_2^{\beta_2} \in C^\infty(\overline{\h_\gamma})$.  Lemma \ref{P:EasyMonomials} says $\bm{B}f = C z_1^{\beta_1} z_2^{-\beta_2}$. Thus $\|f\|_{L^p} <\infty$, while
$\left\|\frac{\dee^{j+l}}{\dee z_1^j\dee z_2^l}\circ\bm{B}f\right\|_{L^p}=\infty$ for those $p$ satisfying \eqref{E:LpBound_jl}.
\end{proof}

\begin{remark}\label{R:Duality}
Theorem \ref{T:DerivsJLBerg} recovers the $L^p_0$ unboundedness range given in Theorem \ref{T:JGALpBoundedness} part (1).  When $j=l=0$, the right hand side of \eqref{E:LpBound_jl} is simply $\rho(m,n)$.  Since $\bm{B}$ is self-adjoint, it must also be unbounded for $1<p<\lambda(m,n)$.
\end{remark}

In particular, Theorem \ref{T:DerivsJLBerg} implies Theorem \ref{T:Main} from the Introduction.

\begin{corollary}\label{C:L2SobFail}
Let $k\ge1$ be an integer and $\gamma = \frac mn \in \Q^+$.  The Bergman projection $\bm{B}$ fails to map $L^2_k(\h_\gamma) \to L^2_k(\h_\gamma)$.
\end{corollary}
\begin{proof}
If $k=j+l\ge1$, then $m(l+1) + n(j+1) - 1 > m+n$.
\end{proof}
\pagebreak

The notion of a {\it lattice point diagram} associated to the domains $\h_\gamma$ was introduced in \cite{EdhMcN17}. The diagrams record exponents of all monomials $z_1^{\alpha_1}z_2^{\alpha_2}\in \cs\left(\h_{m/n},L^p\right)$, as $p$ varies. These diagrams are thus Newton diagrams, but of the entire {\it space} $A^p(\h_\gamma)=\co(\h_\gamma)\cap L^p(\h_\gamma)$ rather than of an individual $f\in A^p(\h_\gamma)$. Several
lattice point diagrams succinctly illustrate Theorem \ref{T:DerivsJLBerg} and Corollary \ref{C:L2SobFail}.

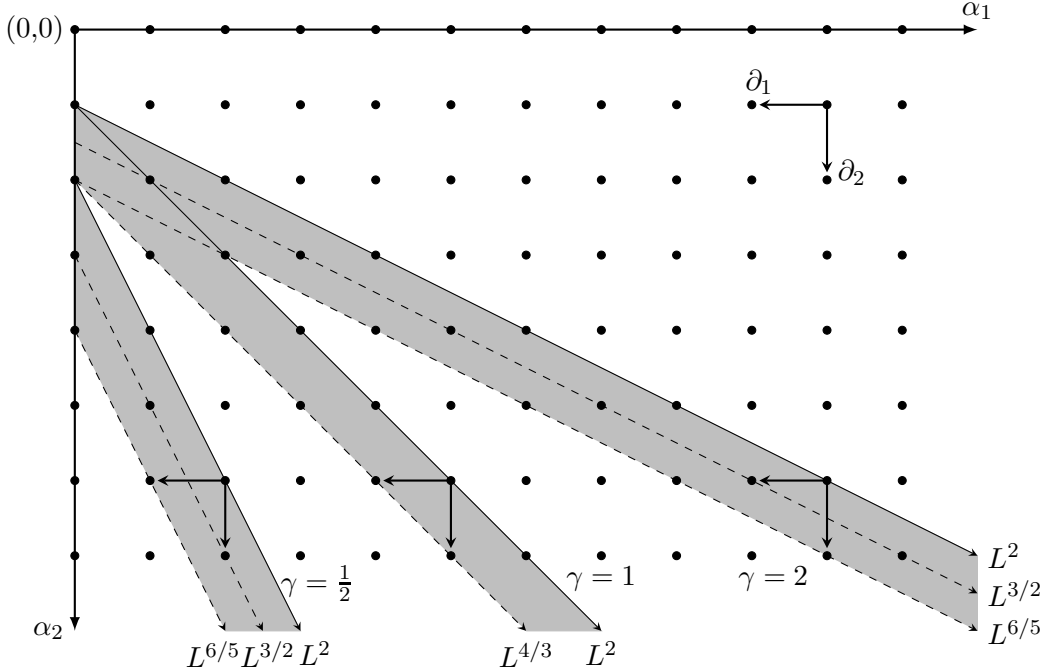
{\begin{center}
\begin{tikzpicture}

\filldraw [lightgray, domain=0:8]  (0,-2) -- (0,-1) -- (7,-8) -- (6,-8);
\filldraw [lightgray, domain=0:8]  (0,-2) -- (0,-1) -- (12,-7) -- (12,-8);
\filldraw [lightgray, domain=0:8]  (0,-4) -- (0,-2) -- (3,-8) -- (2,-8);

\draw[-{latex}, thick] (0,0) -- (12,0) node[anchor=south] {$\alpha_1$};
\draw[-{latex}, thick] (0,0) -- (0,-8) node[anchor=east] {$\alpha_2$};


\draw[-stealth, thin] (0,-1) -- (7,-8) node[anchor=north] {$L^2$};
\draw[-stealth, thin, dashed] (0,-2) -- (6,-8) node[anchor=north] {$L^{4/3}$};

\draw[-stealth, thin] (0,-1) -- (12,-7) node[anchor=west] {$L^2$};
\draw[-stealth, thin, dashed] (0,-1.5) -- (12,-7.5) node[anchor=west] {$L^{3/2}$};
\draw[-stealth, thin, dashed] (0,-2) -- (12,-8) node[anchor=west] {$L^{6/5}$};

\draw[-stealth, thin] (0,-2) -- (3,-8) node[anchor=north] {$\,\,\,\,\,\,L^2$};
\draw[-stealth, thin, dashed] (0,-3) -- (2.5,-8) node[anchor=north] {$\,\,L^{3/2}$};
\draw[-stealth, thin, dashed] (0,-4) -- (2,-8) node[anchor=north] {$L^{6/5}\,\,\,\,\,$};



\draw[-stealth, thick] (2,-6) -- (1.1,-6) ;
\draw[-stealth, thick] (2,-6) -- (2,-6.9) ;

\draw[-stealth, thick] (5,-6) -- (4.1,-6) ;
\draw[-stealth, thick] (5,-6) -- (5,-6.9) ;

\draw[-stealth, thick] (10,-6) -- (9.1,-6) ;
\draw[-stealth, thick] (10,-6) -- (10,-6.9) ;


\draw[-stealth, thick] (10,-1) -- (9.1,-1) node[anchor=south] {$\dee_1$};
\draw[-stealth, thick] (10,-1) -- (10,-1.9) node[anchor=west] {$\dee_2$};

\filldraw[black] (0,0) circle (1.5pt) node[anchor=east] {(0,0)};
\filldraw[black] (1,0) circle (1.5pt) ;
\filldraw[black] (2,0) circle (1.5pt) ;
\filldraw[black] (3,0) circle (1.5pt) ;
\filldraw[black] (4,0) circle (1.5pt) ;
\filldraw[black] (5,0) circle (1.5pt) ;
\filldraw[black] (6,0) circle (1.5pt) ;
\filldraw[black] (7,0) circle (1.5pt) ;
\filldraw[black] (8,0) circle (1.5pt) ;
\filldraw[black] (9,0) circle (1.5pt) ;
\filldraw[black] (10,0) circle (1.5pt) ;
\filldraw[black] (11,0) circle (1.5pt) ;

\filldraw[black] (0,-1) circle (1.5pt) ;
\filldraw[black] (1,-1) circle (1.5pt) ;
\filldraw[black] (2,-1) circle (1.5pt) ;
\filldraw[black] (3,-1) circle (1.5pt) ;
\filldraw[black] (4,-1) circle (1.5pt) ;
\filldraw[black] (5,-1) circle (1.5pt) ;
\filldraw[black] (6,-1) circle (1.5pt) ;
\filldraw[black] (7,-1) circle (1.5pt) ;
\filldraw[black] (8,-1) circle (1.5pt) ;
\filldraw[black] (9,-1) circle (1.5pt) ;
\filldraw[black] (10,-1) circle (1.5pt) ;
\filldraw[black] (11,-1) circle (1.5pt) ;

\filldraw[black] (0,-2) circle (1.5pt) ;
\filldraw[black] (1,-2) circle (1.5pt) ;
\filldraw[black] (2,-2) circle (1.5pt) ;
\filldraw[black] (3,-2) circle (1.5pt) ;
\filldraw[black] (4,-2) circle (1.5pt) ;
\filldraw[black] (5,-2) circle (1.5pt) ;
\filldraw[black] (6,-2) circle (1.5pt) ;
\filldraw[black] (7,-2) circle (1.5pt) ;
\filldraw[black] (8,-2) circle (1.5pt) ;
\filldraw[black] (9,-2) circle (1.5pt) ;
\filldraw[black] (10,-2) circle (1.5pt) ;
\filldraw[black] (11,-2) circle (1.5pt) ;

\filldraw[black] (0,-3) circle (1.5pt) ;
\filldraw[black] (1,-3) circle (1.5pt) ;
\filldraw[black] (2,-3) circle (1.5pt) ;
\filldraw[black] (3,-3) circle (1.5pt) ;
\filldraw[black] (4,-3) circle (1.5pt) ;
\filldraw[black] (5,-3) circle (1.5pt) ;
\filldraw[black] (6,-3) circle (1.5pt) ;
\filldraw[black] (7,-3) circle (1.5pt) ;
\filldraw[black] (8,-3) circle (1.5pt) ;
\filldraw[black] (9,-3) circle (1.5pt) ;
\filldraw[black] (10,-3) circle (1.5pt) ;
\filldraw[black] (11,-3) circle (1.5pt) ;

\filldraw[black] (0,-4) circle (1.5pt) ;
\filldraw[black] (1,-4) circle (1.5pt) ;
\filldraw[black] (2,-4) circle (1.5pt) ;
\filldraw[black] (3,-4) circle (1.5pt) ;
\filldraw[black] (4,-4) circle (1.5pt) ;
\filldraw[black] (5,-4) circle (1.5pt) ;
\filldraw[black] (6,-4) circle (1.5pt) ;
\filldraw[black] (7,-4) circle (1.5pt) ;
\filldraw[black] (8,-4) circle (1.5pt) ;
\filldraw[black] (9,-4) circle (1.5pt) ;
\filldraw[black] (10,-4) circle (1.5pt) ;
\filldraw[black] (11,-4) circle (1.5pt) ;

\filldraw[black] (0,-5) circle (1.5pt) ;
\filldraw[black] (1,-5) circle (1.5pt) ;
\filldraw[black] (2,-5) circle (1.5pt) ;
\filldraw[black] (3,-5) circle (1.5pt) ;
\filldraw[black] (4,-5) circle (1.5pt) ;
\filldraw[black] (5,-5) circle (1.5pt) ;
\filldraw[black] (6,-5) circle (1.5pt) ;
\filldraw[black] (7,-5) circle (1.5pt) ;
\filldraw[black] (8,-5) circle (1.5pt) ;
\filldraw[black] (9,-5) circle (1.5pt) ;
\filldraw[black] (10,-5) circle (1.5pt);
\filldraw[black] (11,-5) circle (1.5pt) ;

\filldraw[black] (0,-6) circle (1.5pt) ;
\filldraw[black] (1,-6) circle (1.5pt) ;
\filldraw[black] (2,-6) circle (1.5pt) ;
\filldraw[black] (3,-6) circle (1.5pt) ;
\filldraw[black] (4,-6) circle (1.5pt) ;
\filldraw[black] (5,-6) circle (1.5pt) ;
\filldraw[black] (6,-6) circle (1.5pt) ;
\filldraw[black] (7,-6) circle (1.5pt) ;
\filldraw[black] (8,-6) circle (1.5pt) ;
\filldraw[black] (9,-6) circle (1.5pt) ;
\filldraw[black] (10,-6) circle (1.5pt) ;
\filldraw[black] (11,-6) circle (1.5pt) ;

\filldraw[black] (0,-7) circle (1.5pt) ;
\filldraw[black] (1,-7) circle (1.5pt) ;
\filldraw[black] (2,-7) circle (1.5pt) ;
\filldraw[black] (3,-7) circle (1.5pt) node[anchor=north] {\,\,\,\,\,\,\,$\gamma = \frac12$};
\filldraw[black] (4,-7) circle (1.5pt) ;
\filldraw[black] (5,-7) circle (1.5pt) ;
\filldraw[black] (6,-7) circle (1.5pt) ;
\filldraw[black] (7,-7) circle (1.5pt) node[anchor=north] {$\gamma = 1$};
\filldraw[black] (8,-7) circle (1.5pt) ;
\filldraw[black] (9,-7) circle (1.5pt) node[anchor=north] {\,\,\,\,\,\,\,\,\,$\gamma = 2$};
\filldraw[black] (10,-7) circle (1.5pt) ;
\filldraw[black] (11,-7) circle (1.5pt) ;
 
\end{tikzpicture}
\end{center}
}

Three  lattice point diagrams on $\h_\gamma$ , corresponding to $\gamma= \frac mn = \frac 12, 1, 2$, are shown.  The indices $\alpha \in \cs\left(\h_{m/n}, L^2\right)$ are exactly those lattice points on and above the line labeled $L^2$ for the corresponding $\gamma$.  The dotted lines, labeled $L^{p}$, are lines parallel to their corresponding $L^2$ lines but passing through the lattice points in $\ell \left(\h_{m/n}, L^{p}\right)$. 
Any lattice point {\it strictly below} the dotted lines correspond to monomials $\notin L^{p}$ for the given $\h_{m/n}$.

Notice that (up to a constant)  $z_1$ derivatives of fourth quadrant monomials are represented by a shift left  and $z_2$ derivatives by a shift down in the lattice point diagram. These operations are labeled 
$\partial_1, \partial_2$ in the diagram. The content of Corollary \ref{C:L2SobFail} is easily seen in this lattice point diagram: monomials on the $L^2$ line are driven below to a corresponding $L^{p}$ line $(p<2)$ by a single application of
 $\partial_1$ or $\partial_2$.  The more precise Theorem \ref{T:DerivsJLBerg} may also be visualized in this way.

\begin{remark} The precise non-isotropic (in terms of derivatives) irregularity in Theorem  \ref{T:DerivsJLBerg} seems noteworthy. The two derivative operations $\partial_1, \partial_2$ are not symmetric with respect to how they drive monomials
out of the boundedness interval $I^p_0$, depending on whether $\gamma >1$ or $\gamma <1$. This is very clear in the diagrams: if $\gamma >1$ (a ``fat Hartogs triangle'' in the
terminology of \cite{Edh16}) more $\partial_1$ derivatives are allowed, while if $\gamma <1$ (a ``thin Hartogs triangle'')  more $\partial_2$ derivatives are allowed.
\end{remark}

\section{Sobolev regularity}\label{S:PositiveMapping}

A class of kernels on the domains $\h_{m/n}$,  containing the Bergman kernel $B_{{m/n}}(z,w)$ and its derivatives, can be analyzed via Lemma \ref{L:GeneralSchur}.  The next result generalizes  Proposition 4.2 of \cite{EdhMcN17}, which required $c=d$.

\begin{lemma}\label{T:TypeCDOperators} 
Let $K:\h_{m/n} \times \h_{m/n} \to \C$ be an integral kernel satisfying
\begin{equation}\label{E:KernelBound}
\left|K(z_1,z_2,w_1,w_2)\right| \lesssim \frac{|z_2|^c|w_2|^d}{|1-z_2\bar{w}_2|^2|z_2^n\bar{w}_2^n - z_1^m\bar{w}_1^m|^2},
\end{equation}
and let $\ck$ be the operator defined by
$\ck (f)(z) := \int_{\h_{m/n}} K(z,w)f(w)\,dV(w)$.

Suppose the following conditions on $c$ and $d$ hold:
\begin{align}\label{CDConditions}
c > 2n\left(1-\frac1m \right) - 2, \quad
d > 2n\left(1-\frac1m \right) - 2, \quad
c + d > 2n\left(2-\frac1m \right) - 2.
\end{align}

Then $\ck: L^p(\h_{m/n}) \to L^p(\h_{m/n})$ is bounded operator for all $p \in (1,\infty)$ satisfying
\begin{equation}\label{E:LpBoundCD}
\frac{2m+2n}{2m+2n +dm -2mn} <p< \frac{2m+2n}{2mn-cm}.
\end{equation}

\begin{remark}\label{R:CDRemark}
If the exponent $c \ge 2n$, the upper bound in \eqref{E:LpBoundCD} can be taken to be $\infty$.  This follows since $|z_2|^c \le |z_2|^{2n}$ for all $z=(z_1,z_2)\in \h_{m/n}$.  Similarly if $d \ge 2n$, the lower bound in \eqref{E:LpBoundCD} is $1$.

Note that the conditions in \eqref{CDConditions} are necessary to ensure the range of $p$ in \eqref{E:LpBoundCD} is a non-degenerate subinterval of $(1,\infty)$.
\end{remark}

\begin{proof}[Proof of Lemma \ref{T:TypeCDOperators}]
Apply Lemma \ref{L:GeneralSchur}, with $h(w) = |w_2|^R(|w_2|^{2n}-|w_1|^{2m})(1-|w_2|^2)$ as the auxiliary function. The parameters $R\ge0$ and $\epsilon\in(0,1)$ are numbers specified later in the proof.  It follows that
\begin{align}
\eqref{E:Schur1} &= \int_{\h_{m/n}} |K(z,w)|h(w)^{-\epsilon}\,dV(w) \notag \\ 
&\lesssim \int_{\h_{m/n}} \frac{|z_2|^c|w_2|^{d-R\epsilon} (|w_2|^{2n}-|w_1|^{2m})^{-\epsilon}(1-|w_2|^2)^{-\epsilon}}{|1-z_2\bar{w}_2|^2 |z_2^n\bar{w}_2^n - z_1^m\bar{w}_1^m|^2} dV(w)  \notag \\
&= \int_{D^*} \frac{|z_2|^c|w_2|^{d-R\epsilon} (1-|w_2|^2)^{-\epsilon}}{|1-z_2\bar{w}_2|^2} \left[ \int_W \frac{(|w_2|^{2n}-|w_1|^{2m})^{-\epsilon}}{|z_2^n\bar{w}_2^n - z_1^m\bar{w}_1^m|^2} dV(w_1) \right] dV(w_2) \label{E:SchurIntEst1},
\end{align}
where $D^*$ is the punctured unit disc and the integral in brackets is taken over the region $W = \{w_1 : |w_1| < |w_2|^{n/m} \}$.  Denote this inner integral by $I$.
\begin{align}
I &= \frac{1}{|z_2|^{2n}|w_2|^{2n+2n\epsilon}}\int_{W} \left(1-\left|\frac{w_1^m}{w_2^n}\right|^2\right)^{-\epsilon}\left|1 - \left(\frac{z_1^m}{z_2^n}\right)\overline{\left(\frac{w_1^m}{w_2^n}\right)} \right|^{-2}dV(w_1) \notag \\
&= \frac{|w_2|^{2n/m-2n-2n\epsilon}}{m|z_2|^{2n}}\int_D \frac{(1-|u|^2)^{-\epsilon}}{|1-z_1^m z_2^{-n}\bar{u}|^2}|u|^{2/m-2}dV(u) \label{E:SchurIntEst2}, 
\end{align}
after the $m$-to-$1$ integral transformation $u = \frac{w_1^m}{w_2^n}$. Lemma \ref{L:IntegralEstimateDisc} yields the estimate
\begin{align}
\eqref{E:SchurIntEst2} &\lesssim \frac{|w_2|^{2n/m-2n-2n\epsilon}}{|z_2|^{2n}} \left(1 -\left| \frac{z_1^m}{z_2^n}\right|^2 \right)^{-\epsilon} \notag \\
&=|z_2|^{2n\epsilon-2n} |w_2|^{2n/m-2n-2n\epsilon} \left(|z_2|^{2n}-|z_1|^{2m}\right)^{-\epsilon}. \label{E:SchurIntEst3}
\end{align}
Now insert \eqref{E:SchurIntEst3} into \eqref{E:SchurIntEst1}:
\begin{align*}
\eqref{E:SchurIntEst1} &\lesssim |z_2|^{c+2n\epsilon -2n} \left(|z_2|^{2n}-|z_1|^{2m}\right)^{-\epsilon} \int_{D^*} \frac{\left(1-|w_2|^2\right)^{-\epsilon}}{|1-z_2 \bar{w}_2|^2} |w_2|^{A}\,dV(w_2),
\end{align*}
where the exponent $A = d + \frac{2n}{m}-2n-(2n+R)\epsilon$ is required to be strictly greater than $-2$ in order for the $D^*$ integral to converge.  This is equivalent to requiring
\begin{equation}\label{E:SetAlpha}
\epsilon < \frac{1}{2n+R}\left(d+\frac{2n}{m}-2n+2\right).
\end{equation}
At this stage, fix $R$ large enough to ensure the right hand side of \eqref{E:SetAlpha} $< 1$. Lemma \ref{L:IntegralEstimateDisc} now applies, since $\epsilon\in(0,1)$.  Doing this yields,
\begin{align*}
\int_{\h_{m/n}} |K(z,w)|h(w)^{-\epsilon}\,dV(w) &\lesssim |z_2|^{c+2n\epsilon -2n} \left(|z_2|^{2n}-|z_1|^{2m}\right)^{-\epsilon} \left(1-|z_2|^{2}\right)^{-\epsilon} \notag \\
&< |z_2|^{-R\epsilon}\left(|z_2|^{2n}-|z_1|^{2m}\right)^{-\epsilon} \left(1-|z_2|^{2}\right)^{-\epsilon}  \\
&= h(z)^{-\epsilon} \notag,
\end{align*}
as long as the exponent $c+2n\epsilon -2n > -R\epsilon$.  This is equivalent to saying
\begin{equation}\label{E:SetBeta}
\epsilon > \frac{2n-c}{2n+R}.
\end{equation}

Inequalities \eqref{E:SetAlpha} and \eqref{E:SetBeta} give the interval $[\alpha,\beta)$ in Lemma \ref{L:GeneralSchur}.  Indeed, it suffices to take 
$\alpha = \frac{2n-c}{2n+R}$ and 
$\beta = \frac{1}{2n+R}\left(d+\frac{2n}{m} - 2n + 2 \right)$.

To generate the interval $[\gamma,\delta)$ needed in Lemma \ref{L:GeneralSchur}, simply switch the roles of $c$ and $d$ in the argument above.  This leads to taking
$\gamma = \frac{2n-d}{2n+R}$ and 
$\delta = \frac{1}{2n+R}\left(c+\frac{2n}{m} - 2n + 2 \right)$.  Lemma \ref{L:GeneralSchur} now gives the claimed result.
\end{proof}
\end{lemma}

\subsection{Mapping of the differentiated projection}

Boundedness of the Bergman projection associated to $\h_{1}$ on the Sobolev space $L^p_1(\h_1)$ can now be given.
In \cite{Edh16}, the Bergman kernel of $\h_{1/n}, n\in\Z^+,$ is computed as
\begin{equation}\label{E:FormulaBergmanThinHartogs}
B_{1/n}(z,w) = \frac{1}{\pi^2} \frac{z_2^n \bar{w}_2^n}{(1-z_2\bar{w}_2)^2(z_2^n\bar{w}_2^n - z_1\bar{w}_1)^2}.
\end{equation}
Throughout the section, subscripts on the projection $\bm{B}_{1/n}$ and the kernel $B_{1/n}(z,w)$ are dropped.

\begin{theorem}\label{T:MappingPropsBergmanDerivatives} On $\h_{1/n}$, $n\in\Z^+$, it holds that
\begin{itemize}
\item[(1)] $\frac{\dee}{\dee z_1}\circ \bm{B}$ maps $L^p_1(\h_{1/n}) \to L^p(\h_{1/n})$ for $p \in \Big(1,\frac{2n+2}{2n}\Big)$.
\item[(2)] $\frac{\dee}{\dee z_2}\circ \bm{B}$ maps $L^p_1(\h_{1/n}) \to L^p(\h_{1/n})$ for $p \in \left(\frac{2n+2}{n+3},2\right)$. 
\end{itemize}
\end{theorem}
\begin{proof}
The spirit is similar to the proof of Theorem \ref{T:BergmanMappingLpSobolevDisc}.  Let $f\in L^p_1(\h_{1/n})$ for $1<p<\infty$, and $j=1,2$.
\begin{align}
\frac{\dee}{\dee z_j}\bm{B}f(z) &= \frac{\dee}{\dee z_j} \int_{\h_{1/n}} B(z,w)f(w)\,dV(w) 
= \frac{1}{z_j} \int_{\h_{1/n}} \bar{w}_j \frac{\dee}{\dee \bar{w}_j} \left( B(z,w) \right) f(w)\,dV(w) \notag \\
&= \frac{1}{z_j} \int_{\h_{1/n}} \ct_{w_j} \left( B(z,w) \right) f(w)\,dV(w), \label{E:ZjDerivHartogs1}
\end{align}
since $B(z,w)$ is anti-holomorphic in $w$.

The $z_1$ and $z_2$ derivatives are handled slightly differently.  Consider the $z_2$ derivative first.  Equation \eqref{E:ZjDerivHartogs1} says
\begin{align}\label{E:FubiniAnnuli}
\frac{\dee}{\dee z_2}\bm{B}f(z) = \frac{1}{z_2} \int_{|w_1|=0}^{|w_1|=1} \left\{ \int_\ca \ct_{w_2} \left(B(z,w) \right) f(w)\,dV(w_2)\right\} dV(w_1),
\end{align}
where the inner integral is over $\ca = \{w_2: |w_1|^{1/n} <|w_2|<1 \}$ for each fixed $w_1$.  Since $\ca$ is an annulus centered at the origin, Proposition \ref{P:IntByParts} transfers the vector field $\ct_{w_2}$ onto $f$ without picking up a boundary integral:
\begin{align}
\eqref{E:FubiniAnnuli} &= -\frac{1}{z_2} \int_{|w_1|=0}^{|w_1|=1} \left\{ \int_\ca  B(z,w) \ct_{w_2}f(w)\,dV(w_2)\right\} dV(w_1) \notag \\
&= -\frac{1}{z_2}  \int_{\h_{1/n}}  B(z,w) \ct_{w_2}f(w)\,dV(w) \notag \\
&=  \int_{\h_{1/n}} \frac{w_2}{z_2} B(z,w) \, \frac{\dee f}{\dee w_2}(w)\,dV(w) -  \int_{\h_{1/n}}  \frac{\bar{w}_2}{z_2} B(z,w) \, \frac{\dee f}{\dee \bar{w}_2}(w)\,dV(w), \label{E:z2SplitInTwoIntegrals}
\end{align}
derivatives interpreted distributionally.  Since $f\in L^p_1(\h_{1/n})$ , $\frac{\dee f}{\dee w_2}, \frac{\dee f}{\dee \bar{w}_2}\in L^p\left(\h_{1/n}\right)$.

By  \eqref{E:FormulaBergmanThinHartogs}, the integral kernels in \eqref{E:z2SplitInTwoIntegrals} satisfy
\begin{align*}
\left| \frac{w_2}{z_2} B(z,w) \right| = \left| \frac{\bar{w}_2}{z_2} B(z,w) \right| \approx \frac{|z_2|^{n-1}|w_2|^{n+1}}{|1-z_2\bar{w}_2|^2|z_2^n\bar{w}_2^n - z_1\bar{w}_1|^2}.
\end{align*}
Therefore Lemma \ref{T:TypeCDOperators}, with $c=n-1$, $d=n+1$, and $m=1$,  shows 
$$\left\|\frac{\dee}{\dee z_2}\circ \bm{B}f\right\|_{L^p\left({\h_{1/n}}\right)}\lesssim \left\|\frac{\dee f}{\dee w_2}\right\|_{L^p}+\left\|\frac{\dee f}{\dee \bar w_2}\right\|_{L^p}\leq \left\|f\right\|_{L^p_1\left({\h_{1/n}}\right)}$$
for $p \in \left(\frac{2n+2}{n+3},2\right)$. This establishes part (2) of the theorem.

Consider the $z_1$ derivative.  Equation \eqref{E:ZjDerivHartogs1} says
\begin{align}\label{E:FubiniDisc}
\frac{\dee}{\dee z_1}\bm{B}f(z) = \frac{1}{z_1} \int_{|w_2|=0}^{|w_2|=1} \left\{ \int_\cd \ct_{w_1} \left(B(z,w) \right) f(w)\,dV(w_1)\right\} dV(w_2),
\end{align}
where the inner integral is taken over $\cd = \{w_1: |w_1|<|w_2|^n \}$ for each fixed $w_2$.  Estimating this term requires more care than was necessary for the $z_2$ derivative. As in the proof of Lemma \ref{T:BergmanMappingLpSobolevDisc}, define a kernel by subtracting from $B(z,w)$ the term $B\big((0,z_2),(0,w_2)\big)$.  Equation \eqref{E:FormulaBergmanThinHartogs} shows
\begin{align}
K(z,w) &:= B(z,w) - B\big((0,z_2),(0,w_2)\big) \notag \\
&= \frac{1}{\pi^2} \left[ \frac{z_2^n \bar{w}_2^n}{(1-z_2\bar{w}_2)^2(z_2^n\bar{w}_2^n - z_1\bar{w}_1)^2} - \frac{1}{z_2^n \bar{w}_2^n (1-z_2\bar{w}_2)^2} \right] \notag \\
&= \frac{1}{\pi^2} \frac{2z_1\bar{w}_1 z_2^n \bar{w}_2^n - z_1^2 \bar{w}_1^2}{z_2^n \bar{w}_2^n (1-z_2\bar{w}_2)^2(z_2^n\bar{w}_2^n - z_1\bar{w}_1)^2}. \label{D:K(z,w)ThisHartogs}
\end{align}

Since $B\big((0,z_2),(0,w_2)\big)$ is independent of $w_1$ and $ \bar{w}_1$, $K(z,w)$ may be substituted for $B(z,w)$ in equation \eqref{E:FubiniDisc}.  Since $\cd$ is a disc centered at the origin, Proposition \ref{P:IntByParts} applies:
\begin{align}
\eqref{E:FubiniDisc} &= -\frac{1}{z_1} \int_{|w_2|=0}^{|w_2|=1} \left\{ \int_\cd  K(z,w) \ct_{w_1}f(w)\,dV(w_1)\right\} dV(w_2) \notag \\
&= -\frac{1}{z_1}  \int_{\h_{1/n}}  K(z,w) \ct_{w_1}f(w)\,dV(w) \notag \\
&=  \int_{\h_{1/n}} \frac{w_1}{z_1} K(z,w) \, \frac{\dee f}{\dee w_1}(w)\,dV(w) -  \int_{\h_{1/n}}  \frac{\bar{w}_1}{z_1} K(z,w) \, \frac{\dee f}{\dee \bar{w}_1}(w)\,dV(w), \label{E:z1SplitInTwoIntegrals}
\end{align}
derivatives interpreted distributionally, as before.  By hypothesis, the functions $\frac{\dee f}{\dee w_1}, \frac{\dee f}{\dee \bar{w}_1}\in L^p\left(\h_{1/n}\right)$.

From \eqref{D:K(z,w)ThisHartogs}, the kernels in \eqref{E:z1SplitInTwoIntegrals} satisfy
\begin{align*}
\left| \frac{w_1}{z_1} K(z,w) \right| = \left| \frac{\bar{w}_1}{z_1} K(z,w) \right| &\approx \left| \frac{w_1}{z_1} \right| \cdot \frac{\left|2z_1\bar{w}_1 z_2^n \bar{w}_2^n - z_1^2 \bar{w}_1^2\right|}{|z_2|^n |\bar{w}_2|^n |1-z_2\bar{w}_2|^2 |z_2^n\bar{w}_2^n - z_1\bar{w}_1|^2} \notag \\
&\lesssim \left| \frac{w_1}{z_1} \right| \cdot \frac{|z_1||w_1||z_2|^n|w_2|^n}{|z_2|^n |{w}_2|^n |1-z_2\bar{w}_2|^2 |z_2^n\bar{w}_2^n - z_1\bar{w}_1|^2}\\
&\le  \frac{|w_2|^{2n}}{|1-z_2\bar{w}_2|^2 |z_2^n\bar{w}_2^n - z_1\bar{w}_1|^2}.
\end{align*}
The last two inequalities hold because $z,w \in \h_{1/n}$.  Lemma \ref{T:TypeCDOperators}, with $c=0$, $d=2n$, and $m=1$, shows 
$$\left\|\frac{\dee}{\dee z_1}\circ \bm{B}f\right\|_{L^p\left({\h_{1/n}}\right)}\lesssim\,\, \left\|f\right\|_{L^p_1\left({\h_{1/n}}\right)}$$
for $p \in \left(1, \frac{2n+2}{2n}\right)$, establishing  part (1) of the theorem.
\end{proof}

\begin{corollary}\label{C:regularity}
The Bergman projection $\bm{B}$ is a bounded operator from $L^p_1(\h_1) \to L^p_1(\h_1)$, for all $\frac{4}{3}<p<2$.
\end{corollary}
\begin{proof} Set $n=1$ in Theorem \ref{T:MappingPropsBergmanDerivatives} and intersect the two intervals of $L^p$ boundedness. It follows that $D\circ \bm{B}$ is $L^p$ bounded for $1<p<2$ for any first derivative $D$. Since $\bm{B}$ itself is $L^p$ bounded for $\frac{4}{3}<p<4$ (Theorem \ref{T:JGALpBoundedness}), the result follows.
\end{proof}

\section{A substitute operator on the Hartogs triangle}\label{S:SubstituteOperators}

In light of Theorem \ref{T:Main}, it is natural to seek operators related to $\bm{B}$ which have better Sobolev mapping behavior than $\bm{B}$ itself.
Pursuing an idea in \cite{ChaEdhMcN18},  a sub-Bergman operator is constructed on $\h_1$ with such improved behavior. $\h_1$ is taken
only for simplicity; the general pattern below extends to other domains.

Consider the set of bounded monomials on $\h_1$:
\begin{equation*}
\cs(\h_1, L^\infty) = \left\{ \alpha=(\alpha_1,\alpha_2) : \alpha_1 \ge 0, \ \ \alpha_1+\alpha_2\ge 0 \right\}.
\end{equation*}
Lemma \ref{L:(p,m/n)-allowable MultiIndices and Norm} shows that $\cs(\h_1, L^\infty)=\cs(\h_1, L^p)$ for $p\geq 4$ and $\cs(\h_1, L^\infty) \subsetneq \cs(\h_1, L^2)$.  Following \cite{ChaEdhMcN18}, define the $L^\infty$ sub-Bergman kernel
\begin{align}\label{E:SubBergKernDef}
\wt{B^\infty}(z,w) := \sum_{\alpha \in \cs(\h_1, L^\infty)} \frac{z^\alpha \bar{w}^\alpha}{\norm{z^\alpha}^2_{L^2(\h_1)}}.
\end{align}
Notice the series in \eqref{E:SubBergKernDef} is only part of the usual series that defines the Bergman kernel. The $L^\infty$ sub-Bergman projection is 
\begin{align}\label{E:SubBergProjDef}
\wt{\bm{B}^{\infty}}f(z) := \int_{\h_1} \wt{B^\infty}(z,w) f(w)\,dV(w)
\end{align}
whenever the integral converges; $f$ is taken from certain $L^p_k(\h_1)$ classes below.  

 A rational expression for \eqref{E:SubBergKernDef} follows from \cite[Proposition 4.33]{ChaEdhMcN18}:
\begin{align*}
\wt{B^\infty}(z,w) = \frac{1}{\pi^2}\frac{2z_2^2\bar{w}_2^2 - z_2^3 \bar{w}_2^3}{(z_2\bar{w}_2 - z_1\bar{w}_1)^2(1-z_2\bar{w}_2)^2}. 
\end{align*}
This immediately yields  the bound
\begin{align}\label{E:BinftyBound}
\left|\wt{B^\infty}(z,w)\right| \lesssim \frac{|z_2|^2|{w}_2|^2}{|z_2\bar{w}_2 - z_1\bar{w}_1|^2|1-z_2\bar{w}_2|^2}.
\end{align}
Lemma \ref{T:TypeCDOperators} with $m=n=1$ and $c=d=2$ shows for each fixed $1<p<\infty$,
\begin{equation}\label{E:LpSubBergmanBoundedness}
\norm{\wt{\bm{B}^\infty}f}_{L^p(\h_1)} \lesssim \norm{f}_{L^p(\h_1)}, \qquad f \in L^p(\h_1).
\end{equation}

Derivatives are now considered.  Mapping properties of $\frac{\dee}{\dee z_2}\circ \wt{\bm{B}^{\infty}}$ may be obtained by following the proof of Theorem \ref{T:MappingPropsBergmanDerivatives} with $\wt{B^\infty}(z,w)$ replacing $B(z,w)$. The steps leading up to \eqref{E:z2SplitInTwoIntegrals} show, for $f \in L^p_1(\h_1)$,
\begin{equation*}
\frac{\dee}{\dee z_2} \wt{\bm{B}^{\infty}}f(z) = \int_{\h_{1}} \frac{w_2}{z_2} \wt{B^\infty}(z,w) \, \frac{\dee f}{\dee w_2}(w)\,dV(w) -  \int_{\h_{1}}  \frac{\bar{w}_2}{z_2} \wt{B^\infty}(z,w) \, \frac{\dee f}{\dee \bar{w}_2}(w)\,dV(w).
\end{equation*}
Thus the operator $\frac{\dee}{\dee z_2}\circ \wt{\bm{B}^{\infty}}$ is controlled  by the kernels
\begin{align}\label{E:z2BinftyBound}
\left|\frac{{w_2}}{z_2}\wt{B^\infty}(z,w)\right| = \left|\frac{\bar{w}_2}{z_2}\wt{B^\infty}(z,w)\right| \lesssim \frac{|z_2||{w}_2|^3}{|z_2\bar{w}_2 - z_1\bar{w}_1|^2|1-z_2\bar{w}_2|^2}.
\end{align}
Lemma \ref{T:TypeCDOperators} (and Remark \ref{R:CDRemark}) with $m=n=1$, $c=1$, $d=3$, shows for each fixed $1<p<4$,
\begin{equation}\label{E:z2LpSubBergmanBoundedness}
\left\|\frac{\dee}{\dee z_2}\circ \wt{\bm{B}^\infty}f\right\|_{L^p\left({\h_{1}}\right)}\lesssim \left\|\frac{\dee f}{\dee w_2}\right\|_{L^p}+\left\|\frac{\dee f}{\dee \bar w_2}\right\|_{L^p}\leq \left\|f\right\|_{L^p_1\left({\h_{1}}\right)}.
\end{equation}

Mapping properties of $\frac{\dee}{\dee z_1}\circ \wt{\bm{B}^{\infty}}$ may be obtained by considering
\begin{align*}
\wt{K^\infty}(z,w) &:= \wt{B^\infty}(z,w) - \wt{B^\infty}\big((0,z_2),(0,w_2)\big) \notag \\
&= \frac{1}{\pi^2} \left[ \frac{2 z_2^2 \bar{w}_2^2 - z_2^3 \bar{w}_2^3}{(1-z_2\bar{w}_2)^2(z_2\bar{w}_2 - z_1\bar{w}_1)^2} - \frac{2 z_2^2 \bar{w}_2^2 - z_2^3 \bar{w}_2^3}{z_2^2 \bar{w}_2^2 (1-z_2\bar{w}_2)^2} \right] \notag \\
&= \frac{1}{\pi^2} \frac{z_1\bar{w}_1 \left(4z_2\bar{w}_2   -   2z_2^2\bar{w}_2^2   -2 z_1\bar{w}_1   + z_1\bar{w}_1 z_2\bar{w}_2 \right)   }{(1-z_2\bar{w}_2)^2(z_2\bar{w}_2 - z_1\bar{w}_1)^2}. 
\end{align*}
Simple estimation shows $\wt{K^\infty}(z,w)$ satisfies a stronger estimate than \eqref{E:BinftyBound}:
\begin{align*}
\left| \wt{K^\infty}(z,w) \right| &\lesssim \frac{|z_1| |{w}_1 | |z_2| |{w}_2 |   }{|1-z_2\bar{w}_2|^2|z_2\bar{w}_2 - z_1\bar{w}_1|^2}.
\end{align*}
Repeating the steps from \eqref{E:FubiniDisc} through \eqref{E:z1SplitInTwoIntegrals} -- with $\wt{K^\infty}(z,w)$ replacing $K(z,w)$ -- shows, for $f \in L^p_1(\h_1)$,
\begin{equation*}
\frac{\dee}{\dee z_1} \wt{\bm{B}^{\infty}}f(z) = \int_{\h_{1}} \frac{w_1}{z_1} \wt{K^\infty}(z,w) \, \frac{\dee f}{\dee w_1}(w)\,dV(w) -  \int_{\h_{1}}  \frac{\bar{w}_1}{z_1} \wt{K^\infty}(z,w) \, \frac{\dee f}{\dee \bar{w}_1}(w)\,dV(w).
\end{equation*}
Thus the operator $\frac{\dee}{\dee z_1}\circ \wt{\bm{B}^{\infty}}$ is controlled by the kernels
\begin{align*}
\left|\frac{{w_1}}{z_1} \wt{K^\infty}(z,w) \right| = \left|\frac{{\bar{w}_1}}{z_1} \wt{K^\infty}(z,w) \right| 
&\lesssim \frac{|z_2| |{w}_2|^3   }{|1-z_2\bar{w}_2|^2|z_2\bar{w}_2 - z_1\bar{w}_1|^2}. 
\end{align*}
This bound is identical to the bound in \eqref{E:z2BinftyBound}.  Consequently, for each fixed $1<p<4$, 
\begin{equation}\label{E:z1LpSubBergmanBoundedness}
\left\|\frac{\dee}{\dee z_1}\circ \wt{\bm{B}^\infty}f\right\|_{L^p\left({\h_{1}}\right)}\lesssim \left\|\frac{\dee f}{\dee w_1}\right\|_{L^p}+\left\|\frac{\dee f}{\dee \bar w_1}\right\|_{L^p}\leq \left\|f\right\|_{L^p_1\left({\h_{1}}\right)}.
\end{equation}

Combining \eqref{E:LpSubBergmanBoundedness}, \eqref{E:z2LpSubBergmanBoundedness}, and \eqref{E:z1LpSubBergmanBoundedness} proves the following
\begin{corollary}\label{C:sub}
$\wt{\bm{B}^\infty}$ maps $L^p_1(\h_1) \to L^p_1(\h_1)$ boundedly for all $1<p<4$.
\end{corollary}

It is not difficult to verify that $\wt{\bm{B}^\infty}$ fails to map $L^p_1(\h_1) \to L^p_1(\h_1)$ for $p\ge4$: take the monomial $f(z) = z_1 \bar{z}_2$ and follow the arguments given in Section \ref{S:NegativeMapping}.
The interested reader is invited to extend Corollary \ref{C:sub} to higher order derivatives.  The statements are

\begin{corollary}
$\wt{\bm{B}^\infty}$ maps $L^p_2(\h_1) \to L^p_2(\h_1)$ boundedly for all $1<p<2$.
\end{corollary}

\begin{corollary}
$\wt{\bm{B}^\infty}$ maps $L^p_3(\h_1) \to L^p_3(\h_1)$ boundedly for all $1<p<\frac{4}{3}$.
\end{corollary}

\begin{remark} Formulas \eqref{E:SubBergKernDef} and \eqref{E:SubBergProjDef} can be modified to define the $L^\infty$ sub-Bergman kernel and projection on a general Reinhardt domain $\cR$.  More generally, for fixed $p \in [2,\infty)$,  $L^p$ sub-Bergman kernels and projections $(\wt{B^p}(z,w)$ and $\wt{\bm{B}^p})$ may be defined on $\cR$ by formulas analogous to \eqref{E:SubBergKernDef}, where the sum is taken over indices $\alpha\in \cs(\cR,L^p)$ -- see \cite[Section 3.6]{ChaEdhMcN18}.

In \cite[Section 4.2.2]{ChaEdhMcN18}, the $\wt{\bm{B}^p}$ are constructed for each $\h_{m/n}$ and shown to stabilize into $m+n$ representatives.  These operators are more regular on $L^p_0$ than $\bm{B}$ is -- see \cite[Theorem 4.3]{ChaEdhMcN18}.  This improved regularity has consequences for holomorphic duality and approximation -- see \cite[Section 4.4]{ChaEdhMcN18}.
\end{remark}

\section*{Acknowledgments}
The authors thank the Erwin Schr\" odinger Institute, Vienna for providing us a fruitful environment for collaboration during a December 2018 workshop.  
The first author also thanks Texas A\&M at Qatar for hosting the stimulating workshop {\em Analysis and Geometry in Several Complex Variables III} in January 2019.

The authors are also grateful to the two anonymous referees, whose comments improved the mathematical and expository content of the paper.

\bibliographystyle{acm}
\bibliography{EdhMcN19}

\begin{thebibliography}{10}

\bibitem{BarSah12}
{\sc Barrett, D., and {\c{S}}ahuto{\u{g}}lu, S.}
\newblock Irregularity of the {B}ergman projection on worm domains in {$\Bbb
  C^n$}.
\newblock {\em Michigan Math. J. 61}, 1 (2012), 187--198.

\bibitem{Bar84}
{\sc Barrett, D.~E.}
\newblock Irregularity of the {B}ergman projection on a smooth bounded domain
  in {${\bf C}^{2}$}.
\newblock {\em Ann. of Math. (2) 119}, 2 (1984), 431--436.

\bibitem{Barrett86}
{\sc Barrett, D.~E.}
\newblock Regularity of the {B}ergman projection and local geometry of domains.
\newblock {\em Duke Math. J. 53}, 2 (1986), 333--343.

\bibitem{Bar92}
{\sc Barrett, D.~E.}
\newblock Behavior of the {B}ergman projection on the {D}iederich-{F}orn\ae ss
  worm.
\newblock {\em Acta Math. 168}, 1-2 (1992), 1--10.

\bibitem{BellLigocka}
{\sc Bell, S., and Ligocka, E.}
\newblock A simplification and extension of {F}efferman's theorem on
  biholomorphic mappings.
\newblock {\em Invent. Math. 57}, 283--289 (1980).

\bibitem{Bell81a}
{\sc Bell, S.~R.}
\newblock Biholomorphic mappings and the {$\bar \partial $}-problem.
\newblock {\em Ann. of Math. (2) 114}, 1 (1981), 103--113.

\bibitem{Boas87}
{\sc Boas, H.~P.}
\newblock The {S}zeg{\H o} projection: {S}obolev estimates in regular domains.
\newblock {\em Trans. Amer. Math. Soc. 300}, 1 (1987), 109--132.

\bibitem{BoaStr91}
{\sc Boas, H.~P., and Straube, E.~J.}
\newblock Sobolev estimates for the {$\overline\partial$}-{N}eumann operator on
  domains in {${\bf C}^n$} admitting a defining function that is
  plurisubharmonic on the boundary.
\newblock {\em Math. Z. 206}, 1 (1991), 81--88.

\bibitem{BoasStraube99}
{\sc Boas, H.~P., and Straube, E.~J.}
\newblock Global regularity of the {$\overline\partial$}-{N}eumann problem: a
  survey of the {$L^2$}-{S}obolev theory.
\newblock In {\em Several complex variables ({B}erkeley, {CA}, 1995--1996)},
  vol.~37 of {\em Math. Sci. Res. Inst. Publ.} Cambridge Univ. Press,
  Cambridge, 1999, pp.~79--111.

\bibitem{ChaEdhMcN18}
{\sc Chakrabarti, D., Edholm, L.~D., and McNeal, J.~D.}
\newblock Duality and approximation of {B}ergman spaces.
\newblock {\em Adv. Math. 341\/} (2019), 616--656.

\bibitem{ChaZey16}
{\sc Chakrabarti, D., and Zeytuncu, Y.}
\newblock ${L}^p$ mapping properties of the {B}ergman projection on the
  {H}artogs triangle.
\newblock {\em Proc. Amer. Math. Soc. 144}, 4 (2016), 1643--1653.

\bibitem{ChaDup06}
{\sc Charpentier, P., and Dupain, Y.}
\newblock Estimates for the {B}ergman and {S}zeg\"o projections on pseudoconvex
  domains of finite type with locally diagonalizable {L}evi form.
\newblock {\em Publ. Math. 50}, 2 (2006), 413--446.

\bibitem{Chen17}
{\sc Chen, L.}
\newblock The {$L^p$} boundedness of the {B}ergman projection for a class of
  bounded {H}artogs domains.
\newblock {\em J. Math. Anal. Appl. 448}, 1 (2017), 598--610.

\bibitem{Chen17c}
{\sc Chen, L.}
\newblock Weighted {S}obolev regularity of the {B}ergman projection on the
  {H}artogs triangle.
\newblock {\em Pacific J. Math. 288\/} (2017), 307--318.

\bibitem{Detraz81}
{\sc Detraz, J.}
\newblock Classes de {B}ergman de fonctions harmoniques.
\newblock {\em Bull. Soc. Math. France 109\/} (1981), 259--268.

\bibitem{DieFor77-1}
{\sc Diederich, K., and Forn{\ae}ss, J.~E.}
\newblock Pseudoconvex domains: an example with nontrivial {N}ebenh\"ulle.
\newblock {\em Math. Ann. 225}, 3 (1977), 275--292.

\bibitem{DurSch04}
{\sc Duren, P., and Schuster, A.}
\newblock {\em Bergman spaces}, vol.~Mathematical surveys and mongraphs.
\newblock American Mathematical Society, 2004.

\bibitem{Edh16}
{\sc Edholm, L.~D.}
\newblock {B}ergman theory of certain generalized {H}artogs triangles.
\newblock {\em Pacific J. Math. 284}, 2 (2016), 327--342.

\bibitem{EdhMcN16}
{\sc Edholm, L.~D., and McNeal, J.~D.}
\newblock The {B}ergman projection on fat {H}artogs triangles: ${L}^p$
  boundedness.
\newblock {\em Proc. Amer. Math. Soc. 144}, 5 (2016), 2185--2196.

\bibitem{EdhMcN17}
{\sc Edholm, L.~D., and McNeal, J.~D.}
\newblock Bergman subspaces and subkernels: degenerate ${L}^p$ mapping and
  zeroes.
\newblock {\em J. Geom. Anal. 27}, 4 (2017), 2658--2683.

\bibitem{ForRud74}
{\sc Forelli, F., and Rudin, W.}
\newblock Projections on spaces of holomorphic functions in balls.
\newblock {\em Ind. Univ. Math. J. 24\/} (1974), 593--602.

\bibitem{HerMcN12}
{\sc Herbig, A.-K., and McNeal, J.}
\newblock A smoothing property of the {B}ergman projection.
\newblock {\em Math. Ann. 354}, 2 (2012), 427--449.

\bibitem{Koenig02}
{\sc Koenig, K.~D.}
\newblock On maximal {S}obolev and {H}\"older estimates for the tangential
  {C}auchy-{R}iemann operator and boundary {L}aplacian.
\newblock {\em Amer. J. Math. 124}, 1 (2002), 129--197.

\bibitem{Kohn99}
{\sc Kohn, J.~J.}
\newblock Quantitative estimates for global regularity.
\newblock In {\em Analysis and geometry in several complex variables ({K}atata,
  1997)}, Trends Math. Birkh\"auser Boston, Boston, MA, 1999, pp.~97--128.

\bibitem{McNeal89}
{\sc McNeal, J.~D.}
\newblock Boundary behavior of the {B}ergman kernel function in {${\bf C}^2$}.
\newblock {\em Duke Math. J. 58}, no. 2 (1989), 499--512.

\bibitem{McNeal91}
{\sc McNeal, J.~D.}
\newblock Local geometry of decoupled pseudoconvex domains.
\newblock In {\em Complex analysis ({W}uppertal, 1991)}, Aspects Math., E17.
  Vieweg, Braunschweig, 1991, pp.~223--230.

\bibitem{McNeal94}
{\sc McNeal, J.~D.}
\newblock Estimates on the {B}ergman kernels of convex domains.
\newblock {\em Adv. Math. 109}, 1 (1994), 108--139.

\bibitem{McNSte94}
{\sc McNeal, J.~D., and Stein, E.~M.}
\newblock Mapping properties of the {B}ergman projection on convex domains of
  finite type.
\newblock {\em Duke Math. J. 73}, 1 (1994), 177--199.

\bibitem{McNSte97}
{\sc McNeal, J.~D., and Stein, E.~M.}
\newblock The {S}zeg\"{o} projection on convex domains.
\newblock {\em Math. Z. 224}, 4 (1997), 519--553.

\bibitem{NagRosSteWai89}
{\sc Nagel, A., Rosay, J.-P., Stein, E.~M., and Wainger, S.}
\newblock Estimates for the {B}ergman and {S}zego kernels in {${\bf C}^2$}.
\newblock {\em Ann. of Math. (2) 129}, 1 (1989), 113--149.

\bibitem{PhoSte77}
{\sc Phong, D.~H., and Stein, E.~M.}
\newblock Estimates for the {B}ergman and {S}zeg\"o projections on strongly
  pseudo-convex domains.
\newblock {\em Duke Math. J. 44}, 3 (1977), 695--704.

\bibitem{RudinFunctiontheory}
{\sc Rudin, W.}
\newblock {\em Function theory in the unit ball in {$\C^n$}}, vol.~241 of {\em
  Grundlehren der Mathematischen Wissenschaften}.
\newblock Springer-Verlag, New York, 1980.

\bibitem{ZahJud64}
{\sc Zaharjuta, V., and Judovic, V.}
\newblock The general form of a linear functional on ${H}_p'$.
\newblock {\em Uspekhi Mat. Nauk. 19}, 2 (1964), 139--142.

\bibitem{Zey13}
{\sc Zeytuncu, Y.}
\newblock ${L}^p$ regularity of weighted {B}ergman projections.
\newblock {\em Trans. Amer. Math. Soc. 365}, no. 6 (2013), 2959--2976.

\bibitem{ZhuBergmanbook}
{\sc Zhu, K.}
\newblock {\em Spaces of holomorphic functions in the unit ball}, vol.~226 of
  {\em Graduate Texts in Mathematics}.
\newblock Springer-Verlag, New York, 2005.

\end{thebibliography}

\end{document}